\documentclass[a4paper,12pt]{article}
\usepackage[margin=1in]{geometry}
\usepackage{float}
\usepackage{complexity}
\usepackage{amsmath}
\usepackage{wrapfig}
\usepackage{enumitem}
\usepackage{amssymb}
\usepackage{amsthm}
\usepackage{mathtools}
\usepackage{tikz}
\usepackage{authblk}
\usepackage{biblatex}
\usepackage{hyperref}
\usetikzlibrary{decorations.pathmorphing}

\newcommand{\N}{\mathbb N}

\theoremstyle{definition}
\newtheorem{lem}{Lemma}
\newtheorem{prop}[lem]{Proposition}
\newtheorem{cor}[lem]{Corollary}

\date{}

\title{Partitioning permutations into monotone subsequences}
\author{David Wärn}
\affil{\texttt{codw2@cam.ac.uk}}

\begin{document}

\maketitle

\begin{abstract}

A permutation is \emph{$k$-coverable} if it can be partitioned into $k$
monotone subsequences. Barber conjectured that, for any given permutation, if
every subsequence of length $k+2 \choose 2$ is $k$-coverable then the
permutation itself is $k$-coverable. This conjecture, if true, would be best
possible.

Our aim in this paper is to disprove this conjecture for all $k \ge 3$. In
fact, we show that for any $k$ there are permutations such that every
subsequence of length at most $(k/6)^{2.46}$ is $k$-coverable while the permutation
itself is not.

\end{abstract}

\section{Introduction}

By \emph{permutation of length $n$}, we mean a permutation $\pi$ of the set
$[n] \coloneqq \{1, \ldots, n\}$ written as a sequence $\pi(1)\ \pi(2)\ \ldots\
\pi(n)$. We say a permutation $\pi$ \emph{contains} another permutation $\tau$,
or $\tau$ is a pattern of $\pi$, if $\pi$ has a subsequence ordered like
$\tau$.

A (possibly empty) subsequence of a permutation is \emph{monotone} if it is
either increasing or decreasing. A permutation is $k$-coverable if it can be
partitioned into $k$ monotone subsequences. If a permutation $\pi$ is
$k$-coverable, then so is every pattern of $\pi$. Barber \cite{barber}
conjectured that, conversely, if every pattern of $\pi$ of length at most $k+2
\choose 2$ is $k$-coverable, then so is $\pi$ itself. This has been verified
for $k \le 2$.

We say $\pi$ is \emph{$k$-critical} if every proper pattern of $\pi$ is
$k$-coverable while $\pi$ itself is not. Barber's conjecture is equivalent to
the assertion that critical permutations have length at most $k+2 \choose 2$. A
simple inductive argument shows that any $k$-critical permutation has length at
least $k+2 \choose 2$. So Barber's conjecture, if true, would be the best
possible.

In this paper, we show that the conjecture is false for every $k \ge 3$. In
fact the length of a $k$-critical permutation need not even be $O(k^2)$ -- we
show that it can be as big as $(k/6)^{2.46}$.

K{\'e}zdy et al \cite{kezdy} showed that $k$-critical permutations cannot be
arbitrarily long. In fact they showed that the length of a $k$-critical
permutation is at most $k^{O(k^6)}$. Feder and Hell \cite{feder} improved this
to show that the length is at most $k^{O(k^2)}$. We present a version of their
argument. This exponential upper bound is currently the best known.

Wagner \cite{wagner} showed that the problem of recognising $k$-coverable
permutations is NP-complete. Thus, unless $\NP = \coNP$, there is generally no
simple reason why a permutation is not $k$-coverable.

Although this is a paper about permutations, the results all have
graph-theoretic analogues. Given a permutation $\pi$, we can construct a graph
-- called a permutation graph -- with vertex set $[n]$ where $i < j$ are
adjacent if $\pi(i) < \pi(j)$. Thus the problem of partitioning a permutation
into increasing and decreasing subsequences is a special case of the problem of
partitioning the vertex set of a graph into cliques and independent sets. This
is called \emph{graph cocolouring} \cite{lesniak}. One special property of
permutation graphs is that they are perfect. As we will mention in the text,
some of our results generalise to perfect graphs, and some to general graphs.

The plan of the paper is as follows. In Section 2 we make some simple remarks,
and then in Section 3 we give our first examples that disprove Barber's
conjecture. In Section 4, which is really the heart of the paper, we show how
to construct long $k$-critical permutations for general $k$. In Section 5, we
show that Barber's conjecture is true in the special case of separable
permutations. In Section 6, we discuss upper bounds. And finally in Section 7
we mention some open problems.

\section{Initial remarks}

For completeness, we start with the following easy fact. It is equivalent to the
assertion that permutation graphs are perfect, and is a special case of Dilworth's
theorem.

\begin{lem}\label{perfect}
Let $\pi$ be a permutation, and suppose the length of the longest increasing
subsequence of $\pi$ is $s$. Then $\pi$ can be partitioned into $s$
decreasing subsequences.

Similarly, if the length of the longest decreasing subsequence of $\pi$ is $r$, then
$\pi$ can be partitioned into $r$ increasing subsequences.
\end{lem}

\begin{proof}
By symmetry, it suffices to prove the first part.

For $1 \le t \le s$, let $D_t$ be the set of $i \in [n]$ such that the longest
increasing subsequence of $\pi$ ending at $\pi(i)$ has length $t$. The
$D_t$ certainly partition $[n]$. Moreover, if $i < j$ and $\pi(i) <
\pi(j)$, then we can get an increasing subsequence ending at $\pi(j)$
by appending $\pi(j)$ to an increasing subsequence ending at $\pi(i)$.
So each $D_t$ must be decreasing, as needed.
\end{proof}

Note that this result is sharp: we cannot hope to partition $\pi$ into less
than $s$ decreasing subsequences. 

We will now use the previous lemma to say something about partitions into
monotone subsequences. For a permutation $\pi$ and nonnegative integers $r, s
\in \N$, we say $\pi$ is $(r,s)$-coverable if $\pi$ can be partitioned into $r$
increasing sequences and $s$ decreasing sequences. Thus $\pi$ is $k$-coverable
iff $\pi$ is $(r, s)$-coverable for some $r+s = k$. We say $\pi$ is $(r,
s)$-critical if every proper pattern is $(r,s)$-coverable while $\pi$ itself is
not.

We write $D(\pi) \subseteq \N^2$ for the set of pairs $(r, s)$ such that $\pi$
is \emph{not} $(r,s)$-coverable. 

\begin{lem}\label{dle}
For any permutation $\pi$ of length $n$, we have $|D(\pi)| \le n$.
\end{lem}

For example, if $n < {k+2 \choose 2}$, then $D(\pi)$ cannot contain all of $\{
(r, s) \in \N^2 \mid r+s \le k\}$, so $\pi$ must be $k$-coverable. Similarly,
if $n < (r+1)(s+1)$, then $D(\pi)$ cannot contain all of $\{0\ldots r\} \times
\{0\ldots s\}$, so $\pi$ must be $(r,s)$-coverable.

\begin{proof}
If $n = 0$, then $D(\pi) = \emptyset$. So assume $n \ge 1$. Let $I$ be a
longest increasing subsequence of $\pi$, and say it has length $k$. Let $\tau$
be the permutation of length $n - k$ ordered like $\pi \setminus I$. Note that
if $\tau$ is $(r, s)$-coverable, then $\pi$ is $(r+1, s)$-coverable. Also, by
\ref{perfect}, $\pi$ is $(0, k)$-coverable.

From this and inductive hypothesis we get that $|D(\pi)| \le |D(\tau)| + k \le
n-k+ k = n$, as desired.
\end{proof}

We remark that the graph-theoretic analogue of Lemma \ref{dle} is true for
perfect graphs by the same proof, but false for general graphs.

Let us write $C(k)$ for the length of the longest $k$-critical permutation, and
$C(r, s)$ for the length of the longest $(r,s)$-critical permutation. We will
later see that $C(k)$ and $C(r,s)$ are always finite.

We certainly have $C(r, s) \ge (r+1)(s+1)$, since any $(r,s)$-critical
permutation is at least this long. Lemma \ref{perfect} shows that we get
equality $C(r, s) = (r+1)(s+1)$ when $r = 0$ or $s = 0$. One can show that
$C(1, 1) = 4 = (1+1)(1+1)$ as follows.  A permutation is $(1,1)$-coverable iff
its graph is a \emph{split} graph -- that is, a graph whose vertex set can be
partitioned into a clique and an independent set.  A graph is split iff it does
not contain a 5-cycle, a 4-cycle, or the complement of a 4-cycle as an induced
subgraph \cite{gibbons}. Finally, no permutation graph contains an induced
5-cycle, since it is not perfect, so $C(1,1) = 4$.

Similarly, we have $C(k) \ge {k+2 \choose 2}$ since any $k$-critical
permutation is at least this long. It is easy to see that $C(0) = 1 = {0+2
\choose 2}$ and $C(1) = 3 = {1 + 2 \choose 2}$. J{\o}rgensen \cite{jorgen}
characterised the graph analogues of 2-critical permutations, and from this
characterisation we get $C(2) = 6 = {2 + 2 \choose 2}$. (This again uses the
fact that permutation graphs do not contain induced long odd cycles.)

As we will see, we never get equality in $C(k) \ge {k+2 \choose 2}$ or $C(r, s)
\ge (r+1)(s+1)$ except for the cases listed above.

\section{Specific examples}

\begin{wrapfigure}{r}{0.27\textwidth}
\vspace{-20pt}
\centering
\begin{tikzpicture}[scale=0.3]
\edef\x{0} \foreach \y in {
10, 5, 1, 7, 11, 4, 9, 2, 6, 12, 8, 3
} { \pgfmathparse{\x+1} \xdef\x{\pgfmathresult}
\path[draw,fill,fill opacity=0.5] (\x, \y) rectangle ++(1,1);
}
\end{tikzpicture}
\vspace{-5pt}
	\caption{$\pi_{12}$}
	\label{twelve}
\vspace{-10pt}
\end{wrapfigure}

We now present some specific examples of critical permutations that are a bit
longer than expected. We will later use these to construct infinite families of
critical permutations, but the exact structure of these starting examples will
be irrelevant. These examples and many more were found by computer search.

One particularly beautiful 3-critical permutation, which we will call
$\pi_{12}$, shows that $C(3) \ge 12$. In numbers, it is 10 5 1 7 11 4 9 2 6 12 8
3, but it is more easily appreciated by drawing the points $(i, \pi(i))$ in the
plane, as pictured.

\begin{wrapfigure}{r}{0.2\textwidth}
\vspace{-10pt}
\centering
\begin{tikzpicture}[scale=0.3]
\edef\x{0} \foreach \y in {
4, 1, 6, 0, 5, 2, 8, 7, 3
} { \pgfmathparse{\x+1} \xdef\x{\pgfmathresult}
\path[draw,fill,fill opacity=0.5] (\x, \y) rectangle ++(1,1);
} \end{tikzpicture}
\vspace{-5pt}
	\caption{$\pi_9$}
	\label{nine}
\vspace{-10pt}
\end{wrapfigure}

It is clear from the picture that $\pi_{12}$ has four-fold rotational symmetry.
This makes it feasible to check by hand that $\pi_{12}$ really is 3-critical.
We note that $\pi_{12}$ is not (0, 3)- or (3, 0)-coverable, since it has
increasing and decreasing subsequences of length 4. It is not (1, 2)-coverable,
since its subsequence 5 1 7 6 12 8 is ordered like 2 1 4 3 6 5 which is not (1,
2)-coverable. Hence by symmetry it is not (2,1)-coverable either. So it is not
3-coverable. On the other hand, if we remove the term $1 = \pi_{12}(3)$, then
we can cover the remaining 11 terms by 5 6 12 and two decreasing subsequences.
If we instead remove $5 = \pi_{12}(2)$ or $7 = \pi_{12}(4)$, then we use 1 4 6
12 and two decreasing subsequences. By symmetry, any pattern of length 11 is
3-coverable, so $\pi_{12}$ is 3-critical.

\begin{wrapfigure}{r}{0.2\textwidth}
\vspace{-10pt}
\centering
\begin{tikzpicture}[scale=0.2]
\edef\x{0} \foreach \y in {
12,14,5,10,3,9,1,7,15,13,11,4,2,8,6
} { \pgfmathparse{\x+1} \xdef\x{\pgfmathresult}
\path[draw,fill,fill opacity=0.5] (\x, \y) rectangle ++(1,1);
} \end{tikzpicture}
\vspace{-5pt}
	\caption{$\pi_{15}$}
	\label{fifteen}
\vspace{-10pt}
\end{wrapfigure}

Our second example, which we will call $\pi_9$, shows that $C(2, 1) \ge 9$. In
numbers, it is 5 2 7 1 6 3 9 8 4. Verifying that $\pi_9$ is (2,1)-critical by
hand would be labourious, but it is easy for a computer program.

Our third and final example, which we will call $\pi_{15}$, shows that $C(2,2)
\ge 15$. In numbers, it is 12 14 5 10 3 9 1 7 15 13 11 4 2 8 6. This
permutation has a clear asymmetry to it: the longest decreasing subsequence has
length 6, but the longest increasing subsequence has length 3. The latter
property will be useful later.

\section{Combining critical permutations}

We will now explain the main idea of this paper: combining critical
permutations to get new critical permutations. The next lemma is the simplest
example of this.

Our basic tools are the direct and skew sums of permutations. Given
permutations $\pi$, $\sigma$ of lengths $n$ and $m$, their direct sum $\pi
\oplus \sigma$ of length $n + m$ is given by the sequence $\pi(1) \ldots
\pi(n), \sigma(1)+n \ldots \sigma(m)+n$. Their skew sum $\pi \ominus \sigma$ is
similarly given by $\pi(1) + m \ldots \pi(n) + m, \sigma(1) \ldots \sigma(m)$.
For example, $(1 3 2) \oplus (2 1) = 1 3 2 5 4$ and $(1 3 2) \ominus (2 1) = 3
5 4 2 1$.

\begin{lem}\label{enkel}
For any $r_1, r_2, s \in \N$, we have $C(r_1 + r_2 + 1, s) \ge C(r_1, s) +
C(r_2, s)$.
Similarly, for $r, s_1, s_2 \in \N$, we have $C(r, s_1 + s_2 + 1) \ge
C(r, s_1) + C(r, s_2)$
\end{lem}

Note that we would get equality above if we replaced $C(r, s)$ with
$(r+1)(s+1)$ throughout.

\begin{proof}
By symmetry, it suffices to prove the first part.

Given an $(r_1, s)$-critical permutation $\pi$ and an $(r_2, s)$-critical
permutation $\sigma$, we claim that $\pi \ominus \sigma$ is $(r_1 + r_2+1,
s)$-critical. The lemma then follows by taking $\pi$, $\sigma$ as long as
possible.

First, suppose $\pi \ominus \sigma$ were $(r_1 + r_2 + 1, s)$-coverable. Each
increasing subsequence of $\pi \ominus \sigma$ lies entirely in either $\pi$
or $\sigma$. We must have at least $r_1 + 1$ increasing subsequences lying in
$\pi$, since $\pi$ is not $(r_1, s)$-coverable, and similarly at least
$r_2+1$ increasing subsequences lying entirely in $\sigma$. This is a
contradiction as $(r_1 + 1) + (r_2 + 1) > r_1+r_2+1$.

Now consider a proper pattern of $\pi \ominus \sigma$. Without loss of
generality it is of the form $\tau \ominus \sigma$ where $\tau$ is a proper
pattern of $\pi$. By assumption, $\tau$ is $(r_1,s)$-coverable. It remains to
check that $\sigma$ is $(r_2+1,s)$-coverable. Pick some term in $\sigma$. By
assumption, $\sigma$ minus this term is $(r_2, s)$-coverable, so by putting in
the term as an increasing subsequence of length 1, $\sigma$ is $(r_2+1,
s)$-coverable, as needed.
\end{proof}

Starting from the fact that $C(2, 1) \ge 9$, repeated application of Lemma
\ref{enkel} shows that $C(r, s) > 1.49(r+1)(s+1)$ for $r, s$ large enough, and
also that $C(r, s) > (r+1)(s+1)$ when $r \ge 2$, $s \ge 1$.

We now look more closely at direct sums, and in particular at $D(\pi \oplus
\sigma)$. First note two properties of $D(\pi) \subseteq \N^2$: it is finite,
and downward closed, in the sense that if $r' \le r$, $s' \le s$, and $(r, s)
\in D(\pi)$, then $(r', s') \in D(\pi)$. We will call a subset of $\N^2$ with
these two properties a \emph{downset}. Our prototypical downset is the
`triangle' $T(k) \coloneqq \{ (r, s) \mid r+s \le k\}$. Note that $\pi$ is
$k$-coverable iff $T(k) \not \subseteq D(\pi)$.

Given downsets $A$, $B$, we let $A \oplus B$ be the downset consisting of $(r,
s)$ such that whenever $s = s_1 + s_2$, either $(r, s_1) \in A$ or $(r, s_2)
\in B$. We can think of this as merging $A$ and $B$ column by column. We have
$|A \oplus B| = |A| + |B|$. Similarly we let $A \ominus B$ consist of $(r, s)$
such that whenever $r = r_1 + r_2$, either $(r_1, s) \in A$ or $(r_2, s) \in
B$. Here $A$ and $B$ are merged row by row, and again $|A \ominus B| = |A| +
|B|$. The following lemma motivates these definitions.

\begin{lem}\label{hom}
For any permutations $\pi$, $\sigma$, we have $D(\pi \oplus \sigma) =
D(\pi) \oplus D(\sigma)$ and $D(\pi \ominus \sigma) = D(\pi) \ominus
D(\sigma)$.
\end{lem}

\begin{proof}
For the first part, simply note that $\pi \oplus \sigma$ is $(r, s)$-coverable
iff, for some $s_1$, $s_2$ with $s = s_1 + s_2$, $\pi$ is $(r, s_1)$-coverable
and $\sigma$ is $(r, s_2)$-coverable. The second part is proved analogously.
\end{proof}

Given a permutation $\pi$ and a downset $A$, we say $\pi$ is $A$-coverable if
$\pi$ is $(r,s)$-coverable for some $(r, s) \in A$. As before, we say $\pi$ is
$A$-critical if every proper pattern of $\pi$ is $A$-coverable while $\pi$
itself is not. Thus $\pi$ is $k$-coverable iff $\pi$ is $T(k)$-coverable, and
$k$-critical iff $\pi$ is $T(k)$-critical.

If $\pi$ is $A$-critical and moreover $A = D(\pi)$, then we say $\pi$ is
$A$-minimal. For example, $\pi$ is $T(k)$-minimal iff $\pi$ is $k$-critical and
also $(r,s)$-coverable whenever $r+s = k+1$. By inspection, $\pi_{12}$ is
$T(3)$-minimal.

Let $M(k)$ be the length of the longest $T(k)$-minimal permutation, so that
$M(k) \le C(k)$ and $M(3) \ge 12$. We introduced the notion of $A$-minimal
permutations in order for the following lemma to work.

\begin{lem}\label{min}
If $\pi$ is $A$-minimal and $\sigma$ is $B$-minimal, then $\pi \oplus \sigma$
is $A \oplus B$-minimal, and $\pi \ominus \sigma$ is $A \ominus B$-minimal.
\end{lem}

\begin{proof}
We have $D(\pi \oplus \sigma) = A \oplus B$ by Lemma \ref{hom}. Now consider a
proper pattern of $\pi \oplus \sigma$, without loss of generality of the form
$\tau \oplus \sigma$ with $\tau$ a proper pattern of $\pi$. By assumption,
$D(\tau)$ is a proper subset of $A$. So $D(\tau \oplus \sigma) = D(\tau) \oplus
D(\sigma) = D(\tau) \oplus B$ is a proper subset of $A \oplus B$, as needed.
\end{proof}

The following lemma lets us prove lower bounds on $M(k)$ and hence on $C(k)$.

\begin{lem}\label{epic}
For any $k \in \N$, we have
\begin{align*}
M(k+1) &\ge M(k) + k+2, \\
M(2k+2)&\ge 3M(k) + M(k+1) \textup{, and}\\
M(2k+3)&\ge M(k) + 3M(k+1).
\end{align*}
\end{lem}

\begin{wrapfigure}{r}{0.35\textwidth}
\vspace{-20pt}
\centering
\begin{tikzpicture}[scale=0.4]
\foreach \x in { (0,0), (1,0), (0.5,{sin(60)})}
\foreach \y in { (0,0), (3,0), (1.5,{3*sin(60)})}
\path[draw,fill=red,fill opacity=0.5] {\x+\y} circle (0.5);
\foreach \x in {
	(2,0), (1,{2*sin(60)}), (2,{2*sin(60)}), (3,{2*sin(60)}), (1.5,{1*sin(60)}), (2.5,{1*sin(60)}) }
\path[draw,fill=blue,fill opacity=0.5] \x circle (0.5);
\foreach \x in { (0,0), (1,0), (2,0), (0.5,{sin(60)}), (1.5,{sin(60)}), (1,{2*sin(60)})}
\foreach \y in { (0,0), (3,0), (1.5,{3*sin(60)})}
\path[draw,fill=blue,fill opacity=0.5,shift={(7,0)}] {\x+\y} circle (0.5);
\foreach \x in { (2,{2*sin(60)}), (3,{2*sin(60)}), (2.5,{1*sin(60)}) }
\path[draw,fill=red,fill opacity=0.5, shift={(7,0)}] \x circle (0.5);
\end{tikzpicture}
\vspace{-5pt}
\caption{Dividing one triangle into four}
\label{nice}
\vspace{-10pt}
\end{wrapfigure}
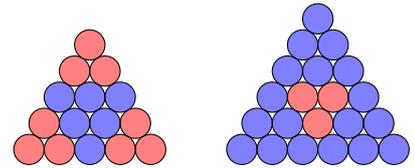

Note that all three inequalities become equalities if we replace $M(k)$ with
$k+2 \choose k$ throughout. For the first inequality, this is immediate. The
latter two are illustrated in Figure \ref{nice}.

\begin{proof}

For the first inequality, we claim that if $\pi$ is $T(k)$-minimal, then $\pi
\ominus (1 \ldots k+2)$ is $T(k+1)$-minimal. This follows from Lemma \ref{min}
together with the fact that $T(k+1) = T(k) \ominus \{0\} \times \{0 \ldots
k+1\}$.

Now suppose $\pi$ is $T(k)$-minimal and $\sigma$ is $T(k+1)$-minimal. We claim
that $(\pi \oplus \pi) \ominus (\pi \oplus \sigma)$ is $T(2k+2)$-minimal and
$(\pi \oplus \sigma) \ominus (\sigma \oplus \sigma)$ is $T(2k+3)$-minimal. By
Lemma \ref{min}, it suffices to prove that $T(2k+2) = (T(k) \oplus T(k))
\ominus (T(k) \oplus T(k+1))$ and $T(2k+3) = (T(k) \oplus T(k+1)) \ominus
(T(k+1) \oplus T(k+1))$. This is straightforward.
\end{proof}

Repeated application of Lemma \ref{epic} starting from $M(3) \ge 12$ gives the
following.

\begin{cor}\label{punkt}
For any $k \ge 3$, we have $M(k) > 1.07{k+2 \choose 2}$, and hence $C(k) >
1.07{k+2 \choose 2}$.
\end{cor}

For the next construction, we need a condition stronger than
`$(r,s)$-critical' but weaker than `$\{0\ldots r\}\times\{0\ldots
s\}$-minimal'. We say $\pi$ is \emph{$(r,s)$-sharp} if $\pi$ is
$(r,s)$-critical and $(0,s+1)$-coverable. We write $S(r, s)$ for the length of
the longest $(r,s)$-sharp permutation, so that $S(r, s) \le C(r, s)$.
By inspection, $\pi_{15}$ is (2,2)-sharp, so $S(2,2) \ge 15$.

\begin{lem}\label{ghee}
For any positive integers $a, b, c, d$, we have 
\[S(ac - 1, bd - 1) \ge S(a-1,
b-1) S(c-1, d-1).\]
\end{lem}

Note that, as before, the inequality becomes an equality if we replace $S(r,
s)$ with $(r+1)(s+1)$.

\begin{proof}

Our tool for proving this lemma is the \emph{tensor product} of permutations.
Given permutations $\pi$ of length $n$ and $\sigma$ of length $m$, their tensor
product $\pi \otimes \sigma$ has length $nm$ and is given explicitly by $(\pi
\otimes \sigma)( (i-1)m + j) = (\pi(i)-1)m + \tau(j)$. We think of this as $n$
copies of $\tau$, ordered like $\pi$. For example, $(1 3 2) \otimes (2 1) = 2 1
6 5 4 3$.

Now it suffices to prove that if $\pi$ is $(a-1, b-1)$-sharp and $\tau$ is
$(c-1, d-1)$-sharp, then $\pi \otimes \tau$ is $(ac-1, bd-1)$-sharp.

There are three things to check. In each case, we use the fact that an
increasing subsequence of $\pi \otimes \tau$ is obtained by picking an
increasing subsequence $I$ of $\pi$ and then an increasing subsequence of
$\tau$ for every copy of $\tau$ corresponding to a point in $I$.

Firstly, we claim that $\pi \otimes \tau$ is $(0, bd)$-coverable. Simply pick
$b$ decreasing subsequences to cover $\pi$ and use each one $d$ times to cover
the corresponding copies of $\tau$.

Secondly, we claim that $\pi \otimes \tau$ is not $(ac-1, bd-1)$-coverable.
Suppose it were. Each copy of $\tau$ must either have at least $c$ increasing
subsequences or $d$ decreasing subsequences passing through it. Thus we get two
sets $C$ and $D$ which cover all of $\pi$. Because $\pi$ is not $(a-1,
b-1)$-coverable, by Lemma \ref{perfect} either $C$ has a decreasing subsequence
of length $a$ or $D$ has an increasing subsequence of length $b$. In the first
case, we get a contradiction as we used less than $ac$ increasing subsequences
to cover each point of $C$ $c$ times. The second case is analogous.

Finally, we claim that any proper subsequence of $\pi \otimes \tau$ is $(ac-1,
bd-1)$-coverable. Suppose we remove a term from the $i$th copy of $\tau$. First
pick a $(0, b)$-covering of $\pi$, and make $d-1$ copies of each of these $b$
subsequences. Then pick an $(a-1, b-1)$-covering of $\pi$ minus the $i$th term,
making 1 copy of each decreasing sequences and $c$ copies of each increasing
sequence. Now each copy of $\tau$ except the $i$th one has either $d-1 + 1 = d$
decreasing sequences through it, or $c$ increasing seqeunces and $d-1$
decreasing sequences; either way it is fully covered. So far we have used
$(a-1)c = ac-c$ increasing sequences and $b(d-1) + b-1 = bd-1$ decreasing
sequences, so we have $c-1$ increasing sequences left. This is precisely enough
to cover the the $i$th copy since it already has $d-1$ decreasing sequences
through it and it is missing one term. \end{proof}

Starting from $S(3-1, 3-1) \ge 15$, repeated application of Lemma \ref{ghee}
gives $S(3^n-1, 3^n-1) \ge 15^n$, and hence $C(3^n-1, 3^n-1) \ge 15^n$. The
following lemma lets us use this to get a lower bound on $C(k)$.

\begin{lem}\label{nio}
For any $r, s, k$, we have $C(r+s) \ge C(r, s)$ and $C(k+1) > C(k)$.
\end{lem}

\begin{proof}
We first show that if $\tau$ is $(r,s)$-critical, then there exists an
$r+s$-critical permutation containing $\tau$ as a pattern. It follows that
$C(r+s) \ge C(r,s)$.

Let $L$ be a skew sum of $r$ copies of $1\ldots N$, and let $R$ be a direct sum
of $s$ copies of $N\ldots 1$, where $N > r+s$. For any permutation $\sigma$, we
claim that $(L \oplus \sigma) \ominus R$ is $r+s$-coverable iff $\sigma$ is
$(r,s)$-coverable.

If $\sigma$ is $(r,s)$-coverable, then we extend the $r$ increasing
subsequences to cover $L$ and the $s$ decreasing subsequences to cover $R$.
Conversely, suppose $(L \oplus \sigma) \ominus R$ is $r+s$-coverable. Each of
copy of $1 \ldots N$ in $L$ must have at least one increasing subsequences
through it, since even $r+s$ decreasing subsequences would not be enough.
Moreover, the $r$ copies cannot share any increasing subsequences, by
construction. So we need at least $r$ increasing subsequences, and similarly at
least $s$ decreasing subsequences. So $(L \oplus \sigma) \ominus R$ must in
fact be $(r, s)$-coverable, and hence $\sigma$ is $(r,s)$-coverable.

Thus if $\tau$ is $(r,s)$-critical, then $(L \oplus \tau) \ominus R$ contains
an $r+s$-critical pattern $\pi$, and $\pi$ must contain $\tau$, as needed.

For the second part, we show that if $\tau$ is $k$-critical, then there exists
a $k+1$-critical permutation containing $\tau$ as a proper pattern. The key is
to observe that for any permutation $\sigma$, the permutation $(1\ldots k+2)
\ominus \sigma$ is $k+1$-coverable iff $\sigma$ is $k$-coverable. Indeed a
$k+1$-covering of $(1\dots k+2) \ominus \sigma$ must use one increasing
sequence to cover $1 \ldots k+2$, and then we are left to $k$-cover $\sigma$.

Thus as before we get a $k+1$-critical permutation $\pi$ containing $\tau$ as a
pattern. We cannot have $\pi = \tau$ since $\tau$ itself is $k+1$-coverable.
\end{proof}

\begin{cor}
For any $n$, $k$, we have
\begin{align*}
C(2\cdot 3^n -2) &\ge 15^n \textup{ and} \\
C(k) &\ge (k/6)^{2.46}
\end{align*}
\end{cor}

\begin{proof} The first inequality follows from $C(2 \cdot 3^n - 2) \ge
C(3^n-1, 3^n-1)$. For the second inequality, pick $n$ such that $k/6 \le 3^n
\le k/2$. Now $C(k) \ge C(2 \cdot 3^n - 2) \ge 15^n = (3^n)^{\log_3 15} \ge
(k/6)^{2.46}$, as needed.
\end{proof}

We remark that $\oplus$, $\ominus$, and $\otimes$ can be defined on arbitrary
graphs. Thus all results in this section remain true in the more general
setting of graph cocolouring, with the same proofs, except for Lemma \ref{ghee}
which only remains true for perfect graphs. If we instead consider the problem
of partitioning the vertex set of a graph into two sets $A$ and $B$ such that
$A$ has no independent set of size $r+1$ and $B$ has no clique of size $s+1$,
then all the results in this section remain true, including Lemma \ref{ghee},
with similar proofs.

\section{Separable permutations}

We were able to construct long critical permutations using skew sums, direct
sums, and tensor products, but only given some base cases. We now explain this
phenomenon.

A permutation is said to be \emph{separable} if it can be obtained from the
length-one permutation $(1)$ by repeatedly taking skew sums and direct sums.
For example, 2 1 4 3 is separable but 3 1 4 2 is not. It can be shown that $\pi
\otimes \sigma$ is separable if $\pi$ and $\sigma$ are, by inducting on $\pi$
and using the right-distributivity $(\pi \oplus \tau) \otimes \sigma = (\pi
\otimes \sigma) \oplus (\tau \otimes \sigma)$, $(\pi \ominus \tau) \otimes
\sigma = (\pi \otimes \sigma) \ominus (\tau \otimes \sigma)$.

We will classify $k$-critical and $(r,s)$-critical separable permutations,
showing that they all have lengths $k+2 \choose 2$ and $(r+1)(s+1)$,
respectively. This explains why our lemmas about $\oplus$, $\ominus$, and
$\otimes$ alone are not enough to get non-trivial lower bounds on $C(k)$ and
$C(r,s)$. The proof crucially relies on our more general notion of being
critical with regard to a downset.

\begin{lem}\label{down}
Given a downset $A$, if a direct sum $\pi \oplus \sigma$ of permutations is
$A$-critical, then $A = B \oplus C$ for some $B$, $C$ such that $\pi$
is $B$-critical and $\sigma$ is $C$-critical.

Similarly, if a skew sum $\pi \ominus \sigma$ is $A$-critical, then $A = B
\ominus C$ for some $B$, $C$ such that $\pi$ is $B$-critical and
$\sigma$ is $C$-critical.
\end{lem}

\begin{proof}
By symmetry, it suffices to prove the first part.

Since $\pi \oplus \sigma$ is not $A$-coverable, we have $A \subseteq D(\pi
\oplus \sigma) = D(\pi) \oplus D(\sigma)$, using Lemma \ref{hom}. It follows
that there exist downsets $B$, $C$ such that $B \subseteq D(\pi)$, $C \subseteq
D(\sigma)$, and $A = B \oplus C$ -- the idea is to start from $B_0 \coloneqq
D(\pi)$, $C_0 \coloneqq D(\sigma)$, and successively remove points $(r,s)$ from
them, starting from larger $r$.

Now we claim that $\pi$ is $B$-critical. We know that $\pi$ is not
$B$-coverable. Moreover, if $B \subseteq D(\tau)$ for some proper pattern
$\tau$ of $\pi$, then $A = B \oplus C \subseteq D(\tau \oplus \sigma)$
contradicting the assumption that every proper pattern of $\pi \oplus \sigma$
if $A$-coverable. So $\pi$ is $B$-critical. Similarly $\sigma$ is $C$-critical,
as needed. \end{proof}

\begin{cor}\label{sepp}
If $\pi$ is separable and $A$-critical, then $\pi$ has length $|A|$.
In particular if $\pi$ is separable and $k$-critical, then $\pi$ has length
$k+2 \choose 2$. If $\pi$ is separable and $(r,s)$-critical, then $\pi$ has length
$(r+1)(s+1)$.
\end{cor}
\begin{proof}
The first claim follows from Lemma \ref{down} and induction on $\pi$. The other
claims are special cases of the first.
\end{proof}

We can say a bit more here. Namely, if $\pi$ is a separable permutation, then
$|D(\pi)|$ is the length of $\pi$, and $\pi$ is $A$-critical iff $A = D(\pi)$.
This gives a procedure to generate all separable $A$-critical permutations:
consider all ways to write $A$ as $B \oplus C$ or $B \ominus C$, and for each
one recursively find the corresponding separable critical permutations.

We remark that if $A \oplus B$ is of the form $\{0\ldots r\} \times \{0 \ldots
s\}$ for some $r$, $s$, then so are $A$ and $B$. Thus if we were only
interested in $(r,s)$-critical separable permutations, then we would not need
to consider general downsets. However, in order to understand $k$-critical
separable permutations, we must consider general downsets.

A graph built from the single-vertex graph using the graph analogues of
$\oplus$ and $\ominus$ is called a \emph{cograph}. It is immediate that a graph
is a cograph iff it corresponds to a separable permutation. Thus Corollary
\ref{sepp} characterises cocolourable cographs in terms of forbidden subgraphs.

\section{Upper bounds}

We now prove that $C(r, s)$ and $C(k)$ are always finite. The following
combinatorial lemma will be key. It can be deduced from the sunflower lemma
\cite{sunflower}, but we will get a slightly better bound using an ad-hoc
argument. The argument we present is essentially the same as the argument used
to prove Theorem 3.1 in \cite{feder}.

\begin{lem}\label{comb}

Fix $r, d \in \N$ with $r \ge 2$ and let $U$ be any nonempty set. Suppose for
every $i \in U$, we have a partial Boolean function $P_i : U \setminus \{i\}
\to [2]$. Suppose the $P_i$ are all close in the sense that for any $i, j \in
U$, there are at most $d$ points $k \in U \setminus\{i,j\}$ where $P_i(k) \ne
P_j(k)$.

Then there exists a least integer $N = N(r, d)$ such that if $|U| > N$, then we can
find a total Boolean function $P : U \to [2]$ such that for any $S \subseteq U$
of size at most $r+1$, there exists $i \in U \setminus S$ such that $P(k) =
P_i(k)$ for all $k \in S$.

In fact $N(r, d) \le \min(4r^{d+1}, (4r)^{d/2+1})$.
\end{lem}

\begin{proof}

We first explain why $N(r, d) \le 4r^{d+1}$.

Fix an arbitrary $i \in U$. The idea is to extend $P_i$ to all of $U$ and then
try change it a bit to get the desired $P$. We show that this works when $U$ is
sufficiently large.

Let $p \in [2]$ be arbitrary. For any $T \subseteq U \setminus \{i\}$ we get a
Boolean function $P_T : U \to [2]$ by taking $P_T(i) = p$, and $P_T(j) = P(j)$
iff $j \not \in T$ for $j \in U \setminus \{i\}$. We may assume that this $P_T$ does
not satisfy the desired properties, i.e. that there exists $S_T \subseteq U$ of
size at most $r+1$ such that no $P_k$ agrees with $P_T$ on all of $S_T$. We
must have $i \in S_T$ or $S_T \cap T \ne \emptyset$ since $P_i$ agrees with
itself. For $j \in S_T \setminus (T \cup \{i\})$, say $T \cup \{j\}$ is a
\emph{child} of $T$.

Let $\mathcal F_p$ be the minimal collection of subsets of $U \setminus \{i\}$
such that $\emptyset \in \mathcal F_p$, and if $T \in \mathcal F_p$ has size at
most $d$, then all its children are also in $\mathcal F_p$. Let $\mathcal T_p =
\bigcup_{T \in \mathcal F_p} T$, and let $\mathcal T = \{i\} \cup \mathcal T_1
\cup \mathcal T_2$.

We claim that $|\mathcal T| \le 4r^{d+1}$. It suffices to prove that $|\mathcal
T_p| < 2r^{d+1}$. We have $|\mathcal T_p| \le |\mathcal F_p|$, since each new
set in $\mathcal F_p$ introduces at most one new element. By construction,
$\mathcal F_p$ has $1$ element of size 0, at most $r$ elements of size $1$,
etc, and at most $r^{d+1}$ elements of size $d+1$. So $|\mathcal F_p| \le 1 + r
+ \ldots r^{d+1} < 2r^{d+1}$, as desired.

Finally, suppose $\mathcal T$ is not all of $U$. Pick $j \in U \setminus
\mathcal T$. Consider $P_j$. Let $p = P_j(i)$. Let $T \in \mathcal F_p$ be
maximal such that $P_j(k) \ne P_i(k)$ for all $k \in T$. Since $P_i$ and $P_j$
are close, $T$ has size at most $d$. By definition of $S_T$, we know that
$P_j(k) \ne P_T(k)$ for some $k \in S_T$. Now this contradicts maximality of
$T$, since the child $T \cup \{k\}$ of $T$ is also in $\mathcal F_p$. This
finishes the proof that $N(r, h) \le 4r^{d+1}$.

Let us now prove that $N(r, h) \le (4r)^{d/2+1}$. We may assume that, for some
$i, j$, the functions $P_i$ and $P_j$ differ in exactly $d$ points, since
otherwise we could take $d$ smaller. Fix such $i$ and $j$ and write $D$ for the set
where they differ. Given $Q : D \to [2]$, we have that $Q$ differs with $P_i$
or $P_j$ in at least $d/2$ points. Say $Q$ differs with $P_i$ in at least $d/2$
points. As before, for $p \in [2]$ and $T \subseteq U \setminus (D \cup
\{i\})$, we define $P_T : U \to [2]$ by $P_T(k) = Q(k)$ for $k \in D$, $P_T(i)
= p$, and $P_T(k) = P_i(k)$ iff $k \not \in T$ for $k \in U \setminus (D \cup
\{i\})$. The rest of the proof proceeds as before: for every $Q$ and $p$ we
construct a tree of sets $T$ where this time the maximum size allowed is
$d/2+1$ instead of $d+1$. The set $\mathcal T \subseteq U$ we end up with has
size at most $d + 2 + 2^{d+1}(r + r^2 + \ldots + r^{\left \lfloor d/2 \right
\rfloor+1}) \le 2^{d+1} \cdot 2r^{d/2+1} = (4r)^{d/2+1}$, as needed.
\end{proof}

In the other direction, we have $N(r, d) \ge r^{d/2}$. To see this, let $U$ be
a rooted tree with $d/2 + 1$ layers where every non-leaf node has $r$ children.
For $i \in U$, define $P_i : U \setminus \{i\} \to [2]$ by taking $P_i(j) = 2$
iff $j$ lies on the path from the root to $i$. The $P_i$ are close since each
one takes the value 2 at most $d/2$ times. Suppose there were a $P : U \to [2]$
with the desired properties. Then $P$ must take the value 2 at the root, since
all $P_i$'s do, and $1$ at all leaves. We get a contradiction by looking at the
deepst vertex where $P$ takes the value 2.

We now explain how Lemma \ref{comb} relates to permutations.

\begin{cor}\label{bound}
For any $r, s$ with $r \ge s$ and $r \ge 2$, we have $C(r, s) \le N(r, 2rs)$.
Hence $C(r, s) \le \min(4r^{2rs+1}, (4r)^{rs+1})$.
\end{cor}

\begin{proof}

Let $\pi$ be an $(r, s)$-critical permutation of length $n$. We
wish to show that $n \le N(r, rs)$. For any $i \in [n]$, $\pi$ minus its $i$th
term is $(r, s)$-coverable, so pick an $(r, s)$-covering. Define $P_i : [n]
\setminus \{i\} \to [2]$ by $P_i(j) = 1$ if $j$ is in an increasing
subsequences and $2$ if it is in a decreasing subsequence. Given $i, j \in
[n]$, suppose there are $2rs+1$ points $k \in [n] \setminus \{i, j\}$ where
$P_i(k) \ne P_j(k)$. Then we can find a set $S \subseteq [n] \setminus \{i,
j\}$ of size $rs+1$ such that, without loss of generality, $P_i(k) = 1$ and
$P_j(k) = 2$ for all $k \in S$. Now $S$ is coverable by $r$ increasing
sequences, since $P_i$ comes from an $(r, s)$-covering, and also by $s$
decreasing sequences, since $P_j$ comes from an $(r, s)$-covering. This is a
contradiction. So the $P_i$ are all close.

Now suppose there is $P : [n] \to [2]$ as in Lemma \ref{comb}. By
construction, there can be no decreasing subsequence of length $r+1$ in $\{k
\in [n] \mid P(k) = 1\}$ and also no increasing subsequence of length $s+1$ in
$\{k \in [n] \mid P(k) = 2\}$. Thus by Lemma \ref{perfect}, $\pi$ is
$(r,s)$-coverable, contradicting the assumption that it is $(r,s)$-critical.
\end{proof}

We remark that the above argument applies not only to permutations, but more
generally to cocolouring perfect graphs.

Now that we have upper bounds on $C(r,s)$, finding upper bounds on $C(k)$ is
easy.

\begin{prop}\label{ub}
For any $k$, we have $C(k) \le \sum_{r+s = k} C(r, s)$. Hence, for $k \ge 3$,
we have $C(k) \le (4k)^{k^2/4 + 2}$.
\end{prop}

\begin{proof}
For the first part, suppose $\pi$ is $k$-critical. Given $r+s = k$, $\pi$
cannot be $(r,s)$-coverable, so it contains some $(r,s)$-critical pattern
$\tau_{r,s}$. Now $\tau_{r,s}$ has length at most $C(r,s)$, so $\pi$ has
a pattern $\sigma$ of length at most $\sum_{r+s=k}C(r,s)$ which contains each
$\tau_{r,s}$ as a pattern. This $\sigma$ cannot be $k$-coverable, since it
is not $(r,s)$-coverable for any $r, s$, so it must be all of $\pi$, as needed.

The second part follows from the first together with Corollary \ref{bound} and
the bound $rs \le k^2/4$.
\end{proof}

Together with Lemma \ref{nio}, this shows that $C(k)$ and $C(r,s)$ have roughly
the same order of growth.

\section{Open questions}

This work leaves many questions unanswered. The main question is whether $C(r, s)$
and $C(k)$ grow polynomially or exponentially.

We also do not know any values of $C(r,s)$ or $C(k)$ exactly except for the
small values mentioned in Section 2. In particular, the values of $C(2, 1)$ and
$C(3)$ are not known. There is no (2,1)-critical permutatation of length 10,
11, or 12, by exhaustive search, so we conjecture that $C(2, 1) = 9$. Careful
analysis of the proof of Lemma \ref{comb} gives an upper bound $C(2, 1) \le
94$, which is an order of magnitude bigger than the conjectured value.

Similarly, there is no 3-critical permutation of length 13, so we conjecture
that $C(3) = 12$. Careful analysis of Proposition \ref{ub} gives an upper
bound $C(3) \le 192$.

\subsection*{Acknowledgements}

I thank my supervisor Imre Leader for many helpful discussions, in particular
related to Section 6 and the structure of this paper. I also thank my friend
Gheehyun Nahm for suggesting the idea for Lemma \ref{ghee}.

\printbibliography

\end{document}